\documentclass[12pt,a4paper]{article}
\usepackage{latexsym,epsfig}
\usepackage{amsmath,amssymb,amsthm}
\usepackage{graphicx}
\usepackage{epstopdf}
\usepackage{url,hyperref}

\usepackage{graphicx}
\usepackage{psfrag}
\usepackage{multicol}
\usepackage{enumerate}

\graphicspath{{.},{figures/}}

\newtheorem{theorem}{Theorem}
\newtheorem{lemma}{Lemma}[section]
\newtheorem{conjecture}[lemma]{Conjecture}
\newtheorem{observation}[lemma]{Observation}
\newtheorem{definition}[lemma]{Definition}
\newtheorem{problem}[lemma]{Problem}
\newtheorem{corollary}[lemma]{Corollary}
\newtheorem{proposition}[lemma]{Proposition}

\newtheorem{remark}[lemma]{Remark}

\begin{document}

\title{Coupon-Coloring and total domination in Hamiltonian planar triangulations }

\author{
  {\bf Zolt\'an L\'or\'ant Nagy}\thanks{The author is supported by the Hungarian Research Grant (OTKA) No. K 120154 and by the J\'anos Bolyai Research Scholarship of the Hungarian Academy of Sciences}\\ \\ 
{\small MTA--ELTE Geometric and Algebraic Combinatorics Research Group}\\
{\small H--1117 Budapest, P\'azm\'any P.\ s\'et\'any 1/C, Hungary}\\
{\small \tt{  nagyzoli@cs.elte.hu}}}

\date{}

\maketitle

\begin{abstract}
We consider the so-called coupon-coloring of the vertices of a graph where
every color appears in every open neighborhood, and our aim is to determine the maximal number of colors in such colorings.
In other words, every color class must be a  total dominating set in the graph and  we study the total domatic number of the graph.  We determine this parameter in every  maximal outerplanar graph, and show that every Hamiltonian maximal planar graph has domatic number at least two, partially answering a conjecture of Goddard and Henning.

\end{abstract}

\section{Introduction}
\label{sec:Intro}
Let $G$ be a graph with no isolated vertices.
 Cockayne, Dawes and Hedetniemi introduced the concept of the \textit{total domatic number} of graphs in \cite{CDH}. 
We call a subset $S$ of the vertex set $V(G)$ of a graph $G$ a \textit{total dominating set}, if every vertex of $G$ is adjacent to at least one vertex from $S$. The topic has a large literature, we refer the reader to the excellent surveys of Henning and Yeo \cite{HY} and Henning \cite{H09} for details on total domination. 
The maximum number of disjoint total dominating sets is called the \textit{total domatic number}. 
 One may see that as a partition of the vertices such that every vertex has a neighbor in each partition class, and the maximal cardinality of partition classes is studied. For that reason, in the forthcoming sections we  call a coloring of a graph \textit{coupon-coloring} if every vertex has neighboring vertices in each color class, for brevity. This name was introduced by  Chen, Kim, Tait and Verstraete \cite{Chen}, while Goddard and Henning called such a coloring \textit{thoroughly distributed} \cite{GH16}. Hence the coupon-coloring number and the total domatic number are equivalent concepts. 

The importance of the case when the total domatic (or coupon-coloring) number equals $2$ relies on the fact that it is closely related to the so called  Property B of hypergraphs defined by Erd\H{o}s \cite{Erdos}. More precisely, the neighborhood hypergraph $\Gamma(G)$ of $G$, consisting of the neighborhoods $N(v)$ of the vertices $v\in V(G)$ as hyperedges has Property B if and only if $G$ has a coupon-coloring with two colors (i.e. the total domatic number is at least two). In addition, the concept has applications in network science, see \cite{Chen}, and also related to the panchromatic $k$-colorings of a hypergraph defined by 
Kostochka and Woodall \cite{KW} 
concerning  vertex  colorings with $k$ colors such that every
hyperedge contains each color.

We note that it is NP-complete to decide whether the total domatic number is at least two for a given graph, and there are graph with arbitrarily large minimal degree and total domatic number $1$. On the other hand, the total domatic number of $k$-regular graphs is of order $\Omega(k/\log(k))$ \cite{Chen}.

The problem of finding  the largest possible number of colors in such colorings has been investigated for several graph families. \cite{Chen, Yi}  We mainly focus on planar graphs, following the footsteps of Goddard and Henning \cite{GH16}, and also of Dorfling,  Hattingh and Jonck \cite{DHJ16}.

Our paper is built up as follows. In Section 2, we extend a theorem of Goddard and Henning \cite{GH16} and characterize those maximal outerplanar graphs that have a coupon-coloring with $2$ colors. Section 3 is devoted to the confirmation of their conjecture in a special case, namely we prove that Hamiltonian maximal planar graphs do have a coupon-coloring with $2$ colors. Finally,  we discuss some related results and open problems in Section 4.


\section{Maximal outerplanar graphs }
\label{sec:Outplan}
The total dominating sets of maximal outerplanar graphs have been studied extensively, see \cite{DHJ17, DHJ16, LZZ}.
In their paper \cite{GH16}, Goddard and Henning studied the total domatic number  for special planar graph families, such as outerplanar graphs.  We will strengthen their result 
in the forthcoming section.

First we introduce some notations and present  some useful observations concerning maximal outerplanar graphs.

\begin{definition} The \textit{weak dual} of a maximal outerplanar graph $G$ is  the graph that has a vertex for every bounded face of the embedding of $G$,
and an edge for every pair of adjacent bounded faces of $G$. The unbounded  face (or simply the boundary of the outerplanar graph) determines a Hamiltonian cycle, and edges not belonging to this cycle are called {\em{chords}}. Note that chords correspond to  edges in the weak dual.  We can assign a distance for each pair of vertices, which is their minimal ordinary distance on the Hamiltonian cycle, i.e. for a cyclically indexed maximal outerplanar graph of order $n$, distance of $v_i$ and $v_j$ is $\min\{j-i, n-j+i  \}$ if $j>i$; but for convenience, we can allow both of $\{j-i, n-j+i  \}$ to represent the distance as we will not consider the induced metric. Vertices are always indexed cyclically.
\end{definition}

\begin{observation}\label{eszrevetel} The {weak dual} of an $n$-vertex maximal outerplanar graph $G$ is tree on $n-2$ vertices with maximal degree at most three.  If an edge of the weak dual divides the tree into two subtrees of size $|T_1|$ and $|T_2|$, then the chord corresponding to this edge defines distance $|T_1|+1$, or equivalently $|T_2|+1$ in the above sense; i.e. the difference of the indices of corresponding vertices in $G$ is either  $|T_1|+1$ or $|T_2|+1$. The maximal number of degree $2$ vertices in $G$, or equivalently, the maximal number of leaves in the weak dual is at most $n/2$.
\end{observation}

Our key constructions are the  so called \textit{sun graphs} --- introduced by Goddard and Henning ---  and  \textit{generalized sun graphs}.

\begin{definition}  \cite{GH16}
Let $G$ be any maximal outerplanar graph of order $n\geq 3$. The { \em sun graph} $M(G)$ of $G$ is the graph obtained from  $G$ in the following way. Take a new vertex $v_e$ for each edge $e=uw$ on the boundary of $G$, and join $v_e$ with $u$ and $w$. The resulting graph will be also a maximal planar graph and has order $2n$. 
\end{definition}

\begin{figure}[h]
\centering
\includegraphics[height=3cm]{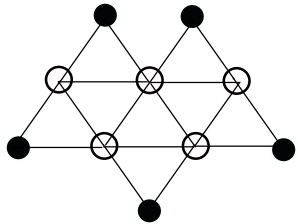}
\caption{The sun graph $M(G_5)$ of the unique maximal planar graph $G_5$ on $5$ vertices}
\label{sun}
\end{figure}

Alternatively, a sun graph is any maximal planar graph  that have half as many degree $2$ vertices as the order of the graph.

\begin{definition}\label{central}
A vertex $v$ of a maximal planar graph is called {\em central vertex} if no degree $2$ vertex is contained in its closed neighborhood, moreover if we start indexing the vertices along the Hamiltonian cycle from that vertex
 (considered as $v=v_0$),  its neighbor vertices  have congruent indices $\pmod 4$. 
\end{definition}

\begin{definition}\label{almost}
A {\em generalized sun graph} is any maximal planar graph of order $n\equiv 2 \pmod 4$ that has half as many degree $2$ vertices plus central vertices as the order of the graph.
\end{definition}

This directly implies  that every second vertex along the unbounded face is either a central vertex or has degree $2$ in a generalized sun graph. Also observe that any edge incident to a central vertex partitions the generalized sun graph to two generalized sun graph of smaller size, by the definition.

Note that the total domination number of maximal outerplanar graphs has been studied by several
authors. Dorfling, Hattingh and  Jonck \cite{DHJ16} showed that, except for two exceptions, every
maximal outerplanar graph with order at least $5$ has total domination number at most $2n/5$.  

This result implies that in general, one cannot obtain a suitable coloring of a maximal outerplanar graph with more than $2$ colors. Goddard and Henning proved the following theorem.

\begin{theorem}\label{GH}\cite{GH16}
Let $n$ be an integer, $n \geq 4$.
If $n\equiv 2 \pmod 4$, then there exists a maximal outerplanar graph of order $n$ without two disjoint total dominating sets. Otherwise (i.e., $n \not\equiv 2 \pmod 4$), every maximal outerplanar graph with order $n$ has two disjoint total dominating sets.
\end{theorem}

The evidence for the first part of Theorem \ref{GH} relies on the fact that for any graph $H$ of order $2k+1$, the corresponding sun graph $M(H)$ of order $4k+2$ admits no  disjoint  pair of total dominating sets, see \cite{GH16}. However, this is not the only family of graphs which does not admit a pair of disjoint total dominating sets. We define below a well structured infinite family as an example and prove that any generalized sun graph also owns this property.

\begin{definition}
A \textit{parasol graph} $\mathrm{P}_n$ with $n=4k+2$ is defined as follows: take $k$-fan on center vertex $u_c$ (i.e. $k$ copies of a triangle joined in a vertex)  on vertices $u_c$, $v_i$ $w_i$ $1\leq i\leq k$ and $k$ copies of the maximal planar graph $G_5$ where $p_iq_i$ denotes the (unique) edge in the $i$th copy which is spanned by the degree $3$ vertices. Identify the pairs $(w_i, v_{i+1})$ for all $1\leq i\leq k-1$ and also the pairs $(v_i, p_i)$, $(w_i, q_i)$ for all  $1\leq i\leq k$. (See figure \ref{parasol}).
\end{definition}

\begin{figure}[h]
\centering
\includegraphics[width=5cm]{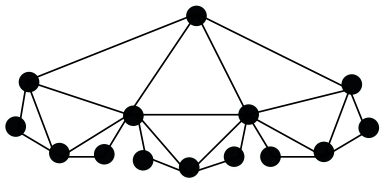}
\caption{The Parasol graph $\mathrm{P}_{14}$ on $n=14$ vertices, a special generalized sun graph}
\label{parasol}
\centering
\end{figure}

The parasol graphs are examples of generalized sun graphs.
A key observation is the following proposition.

\begin{proposition} Generalized sun graphs  do not admit a pair of disjoint total dominating sets.
\end{proposition}

\begin{proof}  The proof is indirect. We show a stronger statement, namely that one cannot color the vertices of generalized sun graphs in such a way that every degree $2$ vertex and central vertex have neighbors from both color class. Take a minimal counterexample $G$ of $4k+2$ vertices with a coupon-coloring with $2$ colors.  Since the unique generalized sun graph for $k=1$ clearly  does not have a coupon-coloring with $2$ colors, $k>1$ must hold.\\
 Index the vertices along the Hamiltonian cycle from $v_1$ to $v_{4k+2}$. By definition,  we can assume that every vertex of odd index is either a vertex of degree $2$ or a central vertex. The cardinality of the vertices implies that there must be two consecutive vertices $v_{2i}$ and $v_{2i+2}$ of even index along the Hamiltonian cycle which have the same color (say white). This would yield a contradiction if $v_{2i+1}$ was a vertex of degree $2$, so it must be a central vertex. Since the coloring was a coupon-coloring, we have an edge between a black vertex $v_{2j}$ and our central vertex $v_{2i+1}$. This edge cuts $G$ into two generalized sun graphs $G'$ (induced by vertices from $v_{2i+1}$ to  $v_{2j}$) and $G''$ (induced by  vertices from $v_{2j}$ to  $v_{2i+1}$) of order less than that of $G$.\\
     However, $G'$ cannot be coupon-colored properly with $2$ colors by the minimality property of $G$, that is, in every $2$-coloring, a vertex of odd index has neighbors from only one color class. Thus this holds for the induced coloring of $G$ on $G'$ as well. But the vertex violating the condition cannot be $v_{2i+1}$ since it has black and white neighbors. On the other hand, the other vertices of odd index in $G'$ has exactly the same neighborhood as in $G$, a contradiction.
\end{proof}

Our contribution is the following characterization theorem.

\begin{theorem}\label{GH_new}
If $G$ is a maximal outerplanar graph of order $n\geq 4$, then either $G$ is a generalized sun graph, or $G$ admits two disjoint total dominating sets.
\end{theorem}

Before we prove this theorem, we introduce some notations and add some useful observations concerning maximal outerplanar graphs.

\begin{definition} Let $G'||G''(x'y', x''y'')$  denote the graph obtained from $G'$ and $G''$ by identifying their edges $x'y'$ and $x''y''$, respectively.
\end{definition}

\begin{lemma}\label{3or4}\cite{GH16} There always exists a chord of length $3$ or $4$ in any maximal outerplanar graph $G$ of order $n\geq 6$.
\end{lemma}
For sake of completeness we add a short proof.

\begin{proof} For $n=6$, the statement is straightforward. Otherwise consider the weak dual $T$ of $G$, and delete its leaves. The remaining graph is a tree $T'$ of order at least two.  Take a chord corresponding to one of the leaves in $T'$. Since this leaf in the remaining tree was attached to $1$ or $2$ leaves before the deletion, the chord is of length $3$ or $4$ via Observation \ref{eszrevetel}.
\end{proof}

\begin{lemma}\label{5+?} In any maximal outerplanar graph $G$ of order $n\geq 5$, there  exists a bounded face such that the deletion of the corresponding vertex in the weak dual graph of $G$ creates at most $1$ tree of order greater than $3$ and at least one of order $2$ or $3$.   
\end{lemma}

\begin{proof} For $n\leq 7$, the statement is easy to verify. Otherwise consider the weak dual $T$ of $G$, label every leaf by $L$, delete them, then label the leaves in the remaining tree $T'$ by $L'$, and delete them as well. In the tree $T''$ we get this way, choose any leaf, and denote it by $u$. (It is easy to see that $T''$ is not empty.)  We prove that the face corresponding to $u$ in the outerplanar graph is a suitable one. Indeed, we deleted at least $1$, at most $2$ neighboring vertices from $u$, and each of these neighbors has at most $2$ neighboring vertices with label $L$.  
\end{proof}

Now we are ready to prove Theorem \ref{GH_new}. 

\begin{proof}[Proof of Theorem \ref{GH_new}]
Notice first that maximal outerplanar graphs contain a Hamiltonian cycle, and if $n\equiv 0 \pmod 4$, then alternating colors in pairs along the Hamiltonian cycle does the job.\\ Also, a suitable coloring for the order  $n\equiv 1, 3 \pmod 4$ can be reduced to the previous case as it was shown in the proof of Theorem \ref{GH}, by Goddard and Henning \cite{GH16}.

Suppose now that $n=4k+2$. We will use induction on $k$ and show that if a maximal outerplanar graph $G$ of order $4k+2$ does not have a coupon-coloring with $2$ colors, then it must be a generalized sun graph.

The case $k=1$ is easy to check. \\
In general, if there exists a chord $xy$ of distance $3$, then there exists a suitable $2$-coloring. Indeed, this means that the chord $e=xy$ divides the graph $G$ to a pair of maximal outerplanar graphs of order $4$ and $4k$, which can both  be colored properly with the condition that $x$ and also $y$ have the same color in both subgraphs.\\
 If there is no chord defining distance $3$, then we can take a chord $e=xy$ of distance $4$ according to Lemma \ref{3or4}. This chords cuts down a maximal outerplanar graph $G_5$ of order $5$. Notice that $x$ and $y$ must be  degree $3$ vertices in $G_5$, otherwise we could find a chord of distance $3$ as well. Moreover, in order for the vertices of degree $2$ in $G_5$ to have a
 proper-colored neighborhood, vertices $x$ and $y$ must have the same color.\\ 
Consider the triangle face adjacent to $e$ which is not contained in $G_5$, and denote it by $xyz$. The deletion of this face of in $G$ yields three maximal outerplanar graph $G'$, $G''$ and the graph $G_5$ we studied. Note that one of $G'$ and $G''$ may be degenerate and have only $1$ edge, but the sum of their order is $4k$. Without loss of generality we may assume that $|G'|\leq |G''|$.

Choose the chord of length $4$ which have the following property:  $|G'|$ is minimal w.r.t. all choice. We claim that we can assume $|G'|\leq 5$ to hold. 
Indeed, this in turn follows from Lemma \ref{5+?}, as the Lemma guarantees a subtree of size $2$ or $3$ and a subtree of size at most $3$ in the weak dual, which are corresponding to a maximal planar graphs of order  $4$ or $5$ and one of order at most $5$ according to Observation \ref{eszrevetel}. 

Suppose  that $xz$ is the edge of $G'$. We end up with the following four cases depending on the order of $G'$.

\begin{figure}[h]
\centering
\includegraphics[height=4cm]{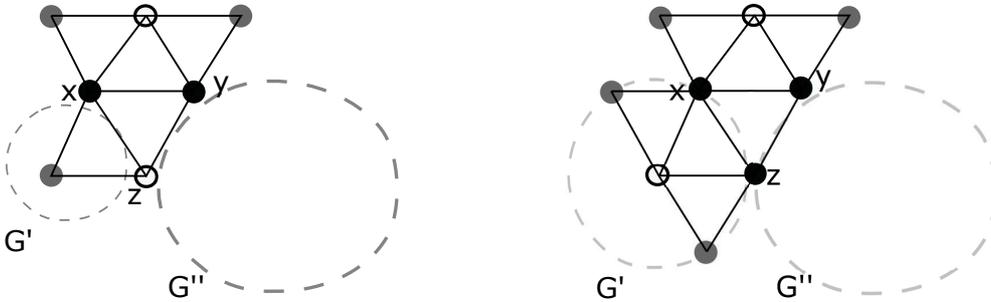}
\caption{Case 2. and Case 4. Gray indicates undefined color}
\label{proof}
\end{figure}

\textbf{Case 1}. $|G'|=2$, the degenerate case. Then $G''$ has order $4k-2$. If $G''$ can be coupon-colored, that means $z$ and $y$ both have neighbors from each color, which makes it easy to extend the coupon-coloring to $G$. Otherwise we can use induction, and get that $G''$ is a generalized sun graph. Hence  $y$ or $z$  is either  a degree $2$ vertex or a central vertex in $G''$. Denote this vertex by  $w_1$ and index the vertices  of $G''$ cyclically from $w_1$ along its outer face.

If $y=w_1$, then we color the vertices in pairs, alternating, along the Hamiltonian cycle of $G''$ starting from $w_2=z$, until $w_{4k-2}$ (white-white-black-black\ldots), and color $y=w_1$ black.
This way we obtain a coloring where only $w_1=y$ has monochromatic neighborhood in $G''$, but $G_5$ can be $2$-coupon colored  to provide $y$ the missing color.

 If $z=w_1$, then it is easy to see that $z$ will be a central vertex in $G$ and $G$ will be a generalized sun graph.

\textbf{Case 2}. $|G'|=3$. In any suitable $2$-coloring, $x$ and $z$ must have different colors, as they are neighbors of a degree $2$ vertex. On the other hand, $G_5$ implies that $x$ and $y$ must have the same color to provide both colors to the degree $2$ vertices. Consequently $G'||G''(zx,zy)$ is  $2$-coupon-colorable if and only if $G$ is  $2$-coupon-colorable. 
But on the other hand, $G'||G''(zx,zy)$ is  a generalized sun graph if and only if $G$ is   a generalized sun graph, hence the claim follows by induction.

\textbf{Case 3.}  $|G'|=4$. This case was excluded before as this implies $xz$ was a chord of distance $3$, which would yield a straightforward proper $2$-coloring.

\textbf{Case 4.} $|G'|=5$.  Here $xz$ must correspond to degree $3$ vertices in $|G'|$, otherwise we could find a chord of distance $3$. It thus follows that in any  $2$-coupon-coloring, $x$ and $z$ must have the same color, similarly to the pair $x$ and $y$. This implies that $G$  has a suitable $2$-coloring if and only if  $G'||G''(zx,zy)$ is $2$-coupon-colorable. But on the other hand, $G'||G''(zx,zy)$ is  a generalized sun graph if and only if $G$ is   a generalized sun graph, hence the claim follows by induction. 
\end{proof}

\section{Hamiltonian maximal planar graphs }
\label{sec:Maxplan}

A central conjecture of Goddard and Henning in \cite{GH16} is the following.

\begin{conjecture}\cite{GH16}\label{sejt} If $G$ is a maximal planar graph (that is, a triangulation) of order at least $4$, then its total domatic number is at least $2$.
\end{conjecture}

They proved that the conjecture holds if every vertex in $G$ has odd degree. Based on our previous section, we can confirm the statement in the Hamiltonian case.

\begin{theorem}\label{maxplanar}  If $G$ is a Hamiltonian maximal planar graph of order at least $4$, then its total domatic number is at least $2$.
\end{theorem}

We first need the following Lemma and Corollary.

\begin{lemma} In every generalized sun graph of order $4k+2$, the number of chords incident to central vertices is at most $k-1$.
\end{lemma}

\begin{proof} For $k=1$, the unique generalized sun graph has no central vertex. In general, choose a chord incident to some central vertex, which has the smallest length. Suppose it is $v_iv_j$, with central vertex $v_i$. It divides the outerplanar graph into two generalized sun graphs, induced by the vertices $[v_i \ldots v_j]$ and $[v_j \ldots v_i]$ (written cyclically). The minimality condition implies that one of them cannot contain further diagonals incident to central vertices. However, the other part has cardinality at most $4(k-1)+2$ by the properties of central vertices, hence the statement follows by induction.
\end{proof}

\begin{remark} The Parasol graph of order $4k+2$ shows that the bound is sharp.
\end{remark}

\begin{corollary}\label{deg2} The number of degree $2$ vertices exceeds the number of central vertices in generalized sun graphs.
\end{corollary}

\begin{proof}[Proof of Theorem \ref{maxplanar}] 
First notice that any Hamiltonian maximal planar graph can be a graph obtained from two maximal outerplanar graphs by identifying their Hamiltonian cycle. Hence essentially we only have to deal with those maximal planar graphs which are obtained from two generalized sun graphs, otherwise the theorem follows from Theorem \ref{GH_new}. 
	 
Assume that this is the case and the order of  the graph is $n=4k+2$.
  We can exclude the case when the union of the set of degree $2$ vertices  and central vertices coincide in the two outerplanar graphs. Indeed, Corollary \ref{deg2} would imply that some degree $2$ vertices also coincide in that case, while every degree must be at least $3$ in a maximal planar graph.
  
We claim that there exists a pair of vertices $v_i$ and $v_{i+3}$ having distance $3$ along the Hamiltonian cycle, such that both of them is a degree $2$ vertex in one of the maximal outerplanar graphs that form $G$. Consider the set of degree $2$ vertices $\{v_j: j\in J\}$  in the first generalized sun graph $G_1$ (for some index set $J$), and take the set consists of vertices $W:=\{v_{j+3}: j\in J\}$.  It also follows from Corollary \ref{deg2} that $W$ and the set of degree $2$ vertices  in the second generalized sun graph $G_2$ cannot be disjoint, proving our claim.

To end the proof, we define a suitable two-coloring of the vertices. Choose a  pair of vertices  having distance $3$ along the Hamiltonian cycle, we may suppose it is $v_1$ and $v_4$. Color a vertex $v_i$ white if $i \equiv 1, 2 \pmod 4$ and $i>4$, otherwise color it black. Then it is obvious that every vertex apart from $v_2$ and $v_3$ has both black and white neighbors along the Hamiltonian cycle. However, since $v_1$ and $v_4$ were degree $2$ vertices in $G_1$ and $G_2$ respectively,  the edges $v_{4k+2}v_2$ and $v_3v_5$  are contained in $G$, ensuring both colors in the neighborhood for $v_2$ and $v_3$ as well.
\end{proof}



\section{Concluding remarks and open questions}
\label{sec:Final}

Dorfling, Goddard, Hattingh and Henning studied the minimum number of edges needed to be added to  graphs of order $n$ with minimum degree at least two to obtain two disjoint total dominating sets  \cite{DGHH}. They proved that asymptotically $(n-2\sqrt{n})$ edges are needed at most in general to do so. 
Observe that Theorem \ref{GH_new}  implies a similar result concerning maximal outerplanar graphs (which also have minimal degree $2$).

\begin{corollary} We can obtain a graph having two disjoint total dominating set  from every maximal outerplanar graph by the addition of at most one edge.
\end{corollary} 

\begin{proof} It is enough to check  the statement for generalized sun graphs in view of Theorem \ref{GH_new}. Consider a degree $2$ vertex $v_i$, and its neighbors $v_{i-1}$ and $v_{i+1}$ lying on the outer face. Since $G$ is maximal outerplanar, $v_{i-1}$ and $v_{i+1}$ are adjacent and they have exactly one common neighbor $w$ beside $v_i$. If we delete  $v_{i-1}v_{i+1}$ and add $wv_i$, the provided graph will be a maximal outerplanar graph, but not a generalized sun graph. Hence, the addition of edge $wv_i$ implies the existence of two disjoint total dominating sets via Theorem \ref{GH_new}. 
\end{proof}

Tthe proof of Theorem \ref{maxplanar} and  its ingredients suggest that under some weak assumptions the total domatic number of maximal planar graphs  exceeds $2$. This opens the way for the following question.

\begin{problem} Describe those maximal planar graphs that have total domatic number  $2$.
\end{problem} 

Hamiltonicity of maximal planar graphs may serve as an intermediate step towards Conjecture \ref{sejt}. Note that a well known result of Whitney  shows that every
maximal planar graph without separating triangles is Hamiltonian, where a \textit{separating triangle}
is a triangle whose removal disconnects the graph. This was even strengthened in the work of Helden who  proved that each maximal
planar graph with at most five separating triangles is Hamiltonian. \cite{Helden}.

\textbf{Acknowledgement}

This work started at the Workshop on Graph and Hypergraph Domination in Balatonalm\'adi in June 2017. The author is thankful to the organizers and Michael Henning for posing the problem, and to  Claire Pennarun, D\"{o}m\"{o}t\"{o}r P\'alv\"{o}lgyi, Bal\'azs Keszegh, Zolt\'an Bl\'azsik and D\'aniel Lenger for the discussion on the topic. Also, grateful acknowledgement is due to Claire Pennarun for his helpful suggestions in order to improve the presentation of the paper.

\end{document}